\documentclass[12pt]{article}

\usepackage{multirow}
\usepackage{amsmath}
\usepackage{amsfonts}
\usepackage{latexsym}
\usepackage{amssymb}
\usepackage{stmaryrd}
\usepackage{overpic}
\usepackage{graphicx}
\usepackage{float}
\usepackage{algorithm}
\usepackage{algorithmic}
\usepackage{xparse}

\NewDocumentCommand{\INTERVALINNARDS}{ m m }{
    #1 {,} #2
}
\NewDocumentCommand{\interval}{ s m >{\SplitArgument{1}{,}}m m o }{
    \IfBooleanTF{#1}{
        \left#2 \INTERVALINNARDS #3 \right#4
    }{
        \IfValueTF{#5}{
            #5{#2} \INTERVALINNARDS #3 #5{#4}
        }{
            #2 \INTERVALINNARDS #3 #4
        }
    }
}

\newtheorem{theorem}{Theorem}[section]

\newtheorem{lemma}[theorem]{Lemma}

\newenvironment{proof}[1][Proof]{\textbf{#1.} }
{\ \rule{0.75em}{0.75em}\smallskip}

\newcommand{\bphi}{\boldsymbol{\phi}}
\newcommand{\bvarphi}{\boldsymbol{\varphi}}
\newcommand{\bnu}{\boldsymbol{\nu}}
\newcommand{\btau}{\boldsymbol{\tau}}

\newcommand{\bsigma}{\boldsymbol{\sigma}}
\newcommand{\bSigma}{\boldsymbol{\Sigma}}
\newcommand{\bI}{\boldsymbol{I}}
\newcommand{\bzero}{\boldsymbol{0}}
\newcommand{\bv}{\boldsymbol{v}}

\newcommand{\bu}{\boldsymbol{u}}

\newcommand{\bn}{\boldsymbol{n}}
\newcommand{\br}{\boldsymbol{r}}
\newcommand{\bw}{\boldsymbol{w}}

\begin{document}

\begin{center}
\Large\bf $C^0$ Discontinuous Galerkin Methods for a

Kirchhoff Plate Contact Problem
\end{center}

\begin{center}
Fei Wang\footnote{Department of Mathematics, Pennsylvania State University,
University Park, PA 16802, USA.  Email: wangfeitwo@163.com}, \quad 
Tianyi Zhang\footnote{Program in Applied Mathematical and Computational Sciences, 
University of Iowa, Iowa City, IA 52242, USA. Email: tianyi-zhang@uiowa.edu},
\quad and \quad
Weimin Han\footnote{Department of Mathematics \& Program in
Applied Mathematical and Computational Sciences, University of
Iowa, Iowa City, IA 52242, USA. Email: weimin-han@uiowa.edu}

\end{center}

\bigskip
\begin{quote}
{\bf Abstract.} Discontinuous Galerkin (DG) methods are considered for solving a plate
contact problem, which is a 4th-order elliptic variational inequality of second kind.
Numerous $C^0$ DG schemes for the Kirchhoff plate bending problem are extended to the 
variational inequality. Properties of the DG methods, such as consistency and stability, 
are studied, and optimal order error estimates are derived.  A numerical
example is presented to show the performance of the DG methods;
the numerical convergence orders confirm the theoretical prediction.


\end{quote}

\section{Introduction}

In this paper, we introduce and analyze some $C^0$ discontinuous Galerkin (DG) methods
for a model 4th-order elliptic variational inequality of second kind.  The model
variational inequality arises in the study of a frictional contact problem for Kirchhoff plates.

\subsection{Discontinuous Galerkin methods}

Discontinuous Galerkin methods are an important family of
nonconforming finite element methods for solving partial differential equations.
We refer to \cite{cockburn00} for a historical
account about DG methods. Discontinuous Galerkin methods use
piecewise smooth yet globally less smooth functions to approximate
problem solutions, and relate the information between two neighboring
elements by numerical traces. The practical interest
in DG methods is due to their flexibility in mesh design and
adaptivity, in that they allow elements of arbitrary shapes,
irregular meshes with hanging nodes, and the discretionary local
shape function spaces. In addition, the increase of the locality
in discretization enhances the degree of parallelizability.

There are basically two approaches to construct DG methods for linear elliptic boundary value problems.
The first approach is through the choice of an appropriate
bilinear form that contains penalty terms to penalize jumps across
neighboring elements to make the scheme stable. The second approach is based on
choosing appropriate numerical fluxes to make the method
consistent, conservative and stable. In \cite{arnold00} and
\cite{arnold02}, Arnold, Brezzi, Cockburn, and Marini provided a
unified error analysis of DG methods for linear elliptic boundary value problems
of 2nd-order and succeeded in building a bridge between these two
families, establishing a framework to understand their properties,
differences and the connections between them. In \cite{wang10},
numerous DG methods were extended for solving elliptic variational
inequalities of 2nd-order, and a priori error estimates were
established, which are of optimal order for linear elements. 
In \cite{wang11}, five discontinuous Galerkin schemes with linear elements for 
solving the Signorini problem were studied, and optimal convergence order was proved. 
The ideas presented in \cite{wang11} were extended to solve a quasistatic contact
problem in \cite{wang14}.

In this paper, we study DG methods to solve an elliptic variational
inequality of 4th-order for the Kirchhoff plates. It is difficult
to construct stable DG methods for such problems because of the higher order
in differentiation and of the inequality form. The major known DG methods for 
the biharmonic equation in the literature are primal DG methods, namely
variations of interior penalty (IP) methods
(\cite{babuska73,baker77,brenner05,engel02,mozolevski3,mozolevski1,mozolevski2,suli07}).
Fully discontinuous IP methods, which cover meshes with hanging
nodes and locally varying polynomial degrees, thus ideally suited for
$hp$-adaptivity, were investigated systematically in
\cite{mozolevski3,mozolevski1,mozolevski2,suli07} for biharmonic
problems. In \cite{engel02}, a $C^0$ IP formulation was introduced
for Kirchhoff plates and quasi-optimal error estimates were
obtained for smooth solutions. Unlike fully discontinuous Galerkin
methods, $C^0$ type DG methods do not ``double" the degrees of
freedom at element boundaries. A rigorous error
analysis was presented in \cite{brenner05} for the $C^0$ IP method
under weak regularity assumption on the solution. A weakness of this
method is that the penalty parameter can not be precisely
quantified {\it a priori}, and it must be chosen suitably large to
guarantee stability. However, a large penalty parameter has a
negative impact on accuracy. Based on this observation, a $C^0$ DG
(CDG) method was introduced in \cite{wells07}, where the stability
condition can be precisely quantified. In \cite{huang10}, a
consistent and stable CDG method, called the LCDG method, was
derived for the Kirchhoff plate bending problem.  The LCDG method can be viewed
as an extension of the LDG method studied in \cite{castillo00,cockburn03}.
We will extend these three methods and additionally propose two more
CDG methods to solve the 4th-order elliptic
variational inequality of second kind.  For 4th-order elliptic
variational inequalities of first kind, some DG methods were developed
in \cite{wang13}; however, no error estimates were derived.
In \cite{brenner12}, a quadratic $C^0$ IP method for Kirchhoff plates problem 
with the displacement obstacle was studied, and errors in the energy norm and 
the $L^{\infty}$ norm are given by $O(h^\alpha)$, where $0.5<\alpha\leq 1$.

\subsection{Kirchhoff plate bending problem}

Let $\Omega\subset\mathbb{R}^{2}$ be a bounded polygonal domain
with boundary $\Gamma$. The boundary value problem of a clamped
Kirchhoff plate under a given scaled vertical load $f\in
L^2(\Omega)$ is (cf.\ \cite{reddy07})
%\begin{equation}
%\left\{
%\begin{array}
%[c]{l}%
%\sum_{i,j=1}^2\mathcal{M}_{ij,ij}(u)+f=0\;\;\text{in }\Omega,\\
%u=\partial_{\bnu}u=0\;\;\text{on }\Gamma,
%\end{array}
%\right.  \label{problem1}%
%\end{equation}
%where
%\[ \mathcal{M}_{ij}(u):=-(1-\kappa)\partial_{ij}u
%   -\kappa\sum_{k=1}^2\partial_{kk}u\delta_{ij},\quad
%    1\leq i, j\le 2,  \]
%$\delta_{ij}$ is the usual Kronecker delta,
\begin{equation}
\left\{
\begin{aligned}
& \bsigma=-(1-\kappa)\nabla ^{2}u-\kappa\,
{\rm tr}(\nabla ^{2}u)\bI\;\;\text{in }\Omega,\\
&-\nabla\cdot(\nabla\cdot\bsigma)=f
\;\;\text{in }\Omega, \\
&u=\partial_{\bnu}u=0\;\;\text{on }\Gamma,
\end{aligned}
\right.  \label{problem2}
\end{equation}
where $\kappa\in(0,0.5)$ denotes the Poisson ratio of an elastic thin plate occupying the
region $\Omega$ and $\bnu$ stands for the unit outward normal vector
on $\Gamma$. $\bI$ is the identity matrix of order $2$ and ${\rm tr}$ is
the trace operation on matrices. Here, $\nabla$ is the usual nabla operator, and
we denote the Hessian of
$v$ by $\nabla^{2}v$, i.e.,
\[ \nabla ^{2}v:=\nabla (\nabla v)
  =\nabla ((\partial _{1}v,\partial_{2}v)^{t})=\left(
\begin{array}{cc}
\partial _{11}v & \partial _{12}v \\
\partial _{21}v & \partial _{22}v%
\end{array}%
\right).
\]
Note that the first equation in \eqref{problem2} can be rewritten as
\begin{equation}\label{moment}
\frac{1}{1-\kappa}\bsigma-\frac{\kappa}{1-\kappa^2}({\rm tr}\bsigma)
   \bI=-\nabla ^{2}u.
\end{equation}
%Then, the problem \eqref{problem1} can be rewritten as

For a vector-valued function
$\bv=(v_1,v_2)^t$ and a matrix-valued function
$\bsigma=(\sigma_{ij})_{2\times2}$, we define their divergence by
\[ \nabla\cdot\bv:=v_{1,1}+v_{2,2},\quad
  \nabla\cdot\bsigma:=(\sigma_{11,1}+\sigma_{21,2},\sigma_{12,1}+\sigma_{22,2})^t. \]
We denote the normal and tangential components of a vector $\bv$ on
the boundary by $v_{\nu}=\bv\cdot\bnu$ and
$\bv_{\tau}=\bv-v_{\nu}\bnu$. Similarly, for a tensor $\bsigma$, we
define its normal component $\sigma_{\nu}=\bsigma\bnu\cdot\bnu$ and
tangential component $\bsigma_{\tau} =\bsigma\bnu-\sigma_{\nu}\bnu$.
We have the decomposition formula
\[ (\bsigma\bnu)\cdot\bv=
  (\sigma_{\nu}\bnu+\bsigma_{\tau})\cdot(v_{\nu}\bnu+\bv_{\tau})
   =\sigma_\nu v_\nu+\bsigma_\tau\cdot\bv_\tau. \]
For two matrices $\btau$ and $\bsigma$, their double dot inner
product and corresponding norm are
$\bsigma:\btau=\sum^2_{i,j=1}\sigma_{ij}\tau_{ij}$ and
$|\btau|=(\btau:\btau)^{1/2}$.

The following result is very useful for the analysis of DG
methods, which can be verified directly through integration by parts.

\begin{lemma}  \label{identities}
Let $D$ be a bounded domain with a Lipschitz boundary $\partial D$.
For a symmetric matrix-valued function $\btau$ and a scalar function
$v$, the following two identities hold
\begin{align*}
\int_{D}v\,\nabla\cdot(\nabla\cdot\btau)\,dx
  & =\int_{D}\nabla^{2}v:\btau\,dx
 -\int_{\partial D}\nabla v\cdot(\btau \bn)\,ds
 +\int_{\partial D}v\,\bn\cdot(\nabla\cdot\btau)\,ds,\\
\int_{D}\nabla^{2}v:\btau\,dx& =-\int_{D}\nabla
v\cdot(\nabla\cdot\btau)\,dx +\int_{\partial D}\nabla v\cdot(\btau \bn)\,ds,
\end{align*}
whenever the terms appearing on both sides of the above identities
make sense. Here $\bn$ is the unit outward normal to $\partial D$.
\end{lemma}

Multiplying the second equation in (\ref{problem2}) by a test
function $v \in H^2_0(\Omega)$ and noticing $v=\partial_{\bnu}v=0$,
we get the following equation by Lemma \ref{identities},
\begin{equation}\label{weak_problem1}
-\int_\Omega \bsigma: \nabla ^{2}v \,dx=\int_\Omega fv \,dx.
\end{equation}
By the definition of $\bsigma$ and (\ref{weak_problem1}), the weak
formulation of problem (\ref{problem2}) can be written as
\begin{equation}
{\rm Find}\, u\in H^2_0(\Omega) :\quad a(u,v)=
(f,v)\quad\forall\,v\in H^2_0(\Omega), \label{weak_problem2}
\end{equation}
where the bilinear form is
\begin{equation}
a(u,v)=\int_\Omega \big[ \Delta u\,\Delta v+(1-\kappa)\,
 (2\,\partial_{12}u\,\partial_{12}v-\partial_{11}u\,
 \partial_{22}v-\partial_{22}u\,\partial_{11}v)\big]\,dx,
\label{bilinear}
\end{equation}
 and the linear form is
\[ (f,v)=\int_\Omega f\,v\,dx. \]

In this paper, we consider a plate frictional contact problem, which is a
4th-order elliptic variational inequality (EVI) of second kind (\cite{duvaut76}).
The Lipschitz continuous boundary $\Gamma$ of the domain $\Omega$ is decomposed into three parts:
$\overline{\Gamma_1}$, $\overline{\Gamma_2}$ and $\overline{\Gamma_3}$ with
$\Gamma_1$, $\Gamma_2$ and $\Gamma_3$ relatively open and mutually
disjoint such that ${\rm meas}(\Gamma_1)>0$.
Then the plate frictional contact problem we consider is:
\begin{equation}
{\rm Find}\ u\in V:\quad a(u,v-u) + j(v) - j(u) \ge (f,v-u)\quad\forall\,v\in V.
\label{EVI}
\end{equation}
Here,
\begin{align*}
V &=\left\{v\in H^2(\Omega):\,v=\partial_{\bnu} v=0\ {\rm on\ }\Gamma _1\right\},\label{V_sp}\\
j(v) &= \int_{\Gamma_3} g\, |v|\, ds.
\end{align*}
This variational inequality
describes a simply supported plate. The plate is clamped on the boundary $\Gamma_1$:
\begin{equation}
    v=\partial_{\bnu} v = 0\ {\rm on\ }\Gamma _1, \label{boundary_1}
\end{equation}
is free on $\Gamma_2$, and is in frictional contact on $\Gamma_3$ with
a rigid foundation; $g$ can be interpreted as a frictional bound.  Applying
the standard theory on elliptic variational inequalities (e.g., \cite{atkinson09, glowinski84}),
we know the problem (\ref{EVI}) has a unique solution $u\in V$.

Let
\[\Lambda=\{\lambda\in L^\infty(\Gamma_3): |\lambda|\leq 1 \ {\rm a.e.\ on}\ \Gamma_3\}.\]
We have the following result (\cite{han02}).

\begin{theorem}
A function $u\in V$ is a solution of \eqref{EVI} if and only if there is a $\lambda\in \Lambda$ such that
\begin{align}
    a(u,v)+\int_{\Gamma_3} g\ \lambda\ v\, ds &= (f,v)\quad   \forall\, v\in V, \label{Lar1}\\
    \lambda\ u &= |u|\quad {\rm a.e. \ on}\ \Gamma_3.\label{Lar2}
\end{align}
\end{theorem}

Throughout the paper, we assume the solution of the problem $(\ref{EVI})$
has the regularity $u\in H^3(\Omega)$. The regularity result $u\in H^3(\Omega)$ is shown for solutions of
some variational inequalities of 4th-order (\cite[pp.\ 323--327]{glowinski81}). In error analysis of numerical solutions for
the problem (\ref{EVI}), we need to take advantage of pointwise relations satisfied by the
solution $u$.

Note that $\bsigma$ is defined by the first equation of (\ref{problem2}).  Then $\bsigma\in
H^1(\Omega)^{2\times 2}$.  We rewrite (\ref{EVI}) as
\[ \int_\Omega\left[-\bsigma:\nabla^2(v-u)-f\,(v-u)\right]
   dx + \int_{\Gamma_3} g|v|\,ds - \int_{\Gamma_3} g|u|\,ds \ge 0\quad\forall\,v\in V. \]
Take $v=u\pm\varphi$ for any $\varphi\in C^\infty_0(\Omega)$ to obtain
\[ -\int_\Omega \bsigma:\nabla^2\varphi\,dx
   =\int_\Omega f\,\varphi\,dx\quad\forall\,\varphi\in C^\infty_0(\Omega). \]
Thus,
\[ -\nabla\cdot(\nabla\cdot\bsigma)=f
  \quad{\rm in\ the\ sense\ of\ distribution}. \]
Since $f\in L^2(\Omega)$, we deduce that $\nabla\cdot
(\nabla\cdot\bsigma)\in L^2(\Omega)$ and
\begin{align}\label{relation1}
 -\nabla\cdot(\nabla\cdot\bsigma)=f  \quad{\rm a.e.\ in\ }\Omega.
\end{align}

Since $\nabla\cdot\bsigma\in L^2(\Omega)^2$ and
$\nabla\cdot(\nabla\cdot\bsigma)\in L^2(\Omega)$, we can define
$(\nabla\cdot\bsigma)\cdot\bnu\in H^{-1/2}(\Gamma)$ and it satisfies the relation
\begin{align}
    \langle (\nabla\cdot\bsigma)\cdot\bnu,v \rangle_{1/2,\Gamma}
= \int_\Omega \nabla\cdot(\nabla\cdot\bsigma)\,v\,dx
+ \int_\Omega (\nabla\cdot\bsigma)\cdot\nabla v\,dx \quad\forall\,v\in H^1(\Omega). \label{relation2}
\end{align}
Therefore, for any $v\in H^2(\Omega)$,
\begin{align*}
-\int_\Omega \nabla\cdot(\nabla\cdot\bsigma) \,v\,dx
& =\int_\Omega (\nabla\cdot\bsigma)\cdot\nabla v\,dx
- \langle (\nabla\cdot\bsigma)\cdot\bnu,v \rangle_{1/2,\Gamma}\\
&=  -\int_\Omega \bsigma:\nabla^2v\,dx + \int_\Gamma (\bsigma\bnu)\cdot\nabla v\,ds
- \langle (\nabla\cdot\bsigma)\cdot\bnu,v \rangle_{1/2,\Gamma},
\end{align*}
i.e.,
\[ a(u,v)=\int_\Omega f\,v\,dx   -\int_\Gamma (\bsigma\bnu)\cdot\nabla v\,ds
   + \langle (\nabla\cdot\bsigma)\cdot\bnu,v \rangle_{1/2,\Gamma} \quad\forall\,v\in H^2(\Omega). \]
Recalling the equation \eqref{Lar1}, we then have for any $v\in V$,
\begin{equation}
-\int_\Gamma (\bsigma\bnu)\cdot\nabla v\,ds + \langle(\nabla\cdot\bsigma)\cdot\bnu,v\rangle_{1/2,\Gamma}
+ \int_{\Gamma_3} g\, \lambda\, v\, ds = 0.\label{relation3}
\end{equation}
Let $\sigma_{\bnu}$ and $\sigma_\tau$ be the normal and tangential components of the
vector $\bsigma\bnu$ on $\Gamma$.
In (\ref{relation3}), taking $v\in V$ such that $v=0$ on $\Gamma$ and $\partial_{\bnu} v$
arbitrary on $\Gamma_2\cup\Gamma_3$, we have
\begin{align}\label{relation4}
 \sigma_{\bnu}=0 \quad {\rm a.e.\ on}\ \Gamma_2\cup \Gamma_3
\end{align}
Then from \eqref{relation3} we get
\begin{equation}
-\int_{\Gamma_2\cup \Gamma_3} \sigma_\tau \partial_\tau v\,ds
 + \langle (\nabla\cdot\bsigma)\cdot\bnu,v \rangle_{1/2,\Gamma}
+ \int_{\Gamma_3} g\, \lambda\, v\, ds = 0 \quad \forall\, v\in V.\label{relation5_0}
\end{equation}
Note that the closure of $V$ in $H^1(\Omega)$ is
\[H^1_{\Gamma_1}(\Omega)=\{v\in H^1(\Omega): v=0 \ {\rm a.e.\ on}\ \Gamma_1 \}.\]
Denote
\[\widetilde{H}^1_{\Gamma_1}(\Omega)=\{v\in H^1_{\Gamma_1}(\Omega):\partial_\tau v\in L^2(\Gamma) \}.\]
Then from \eqref{relation5_0}, we conclude that
\begin{equation}
-\int_{\Gamma_2\cup \Gamma_3} \sigma_\tau \partial_\tau v\,ds
+ \langle (\nabla\cdot\bsigma)\cdot\bnu,v \rangle_{1/2,\Gamma}
+ \int_{\Gamma_3} g\, \lambda\, v\, ds = 0 \quad \forall\,
v\in \widetilde{H}^1_{\Gamma_1}(\Omega).\label{relation5}
\end{equation}

The rest of the paper is organized as follows. In Section \ref{sec:2}, we present
some notations, introduce some $C^0$ discontinuous Galerkin
methods for solving the Kirchhoff plate bending problem and
extend them to solve the elliptic variational inequality of
4th-order. In Section \ref{sec:3}, consistency of the CDG methods,
boundedness and stability of the bilinear forms are presented.
A priori error analysis for these CDG methods is established in Section \ref{sec:4}.
In the final section, we report simulation results from a numerical example.

\section{DG methods for Kirchhoff plate problem}\label{sec:2}
\setcounter{equation}0

\subsection{Notations}
Here we introduce some notations to be used later. For a given function
space $B$, let $(B)_s^{2\times2}:=\left\{\btau\in(B)^{2\times2}:
\btau^t=\btau\right\}$. Given a bounded set $D\subset \mathbb{R}^2$
and a positive integer $m$, $H^m(D)$ is the usual Sobolev space with
the corresponding norm $\|\cdot\|_{m,D}$ and semi-norm
$|\cdot|_{m,D}$, which are abbreviated by $\|\cdot\|_{m}$ and
$|\cdot|_{m}$, respectively, when $D$ is chosen as $\Omega$.
$\|\cdot\|_D$ is the norm of the Lebesgue space $L^2(D)$. We assume
$\Omega$ is a polygonal domain and denote by $\{\mathcal{T}_h\}_h$ a
family of triangulations of $\overline{\Omega}$, with the minimal
angle condition satisfied. Denote $h_K ={\rm diam}(K)$ and $h =
\max\{h_K: K\in \mathcal{T}_h\}$. For a triangulation $\mathcal{T}_h$,
let $\mathcal{E}_h$ be the set of all the element edges, $\mathcal{E}_h^b$
the set of all the element edges that lie on the boundary $\Gamma$,
$\mathcal{E}^i_h := \mathcal{E}_h \backslash \mathcal{E}_h^b$ the set of
all interior edges, and $\mathcal{E}^0_h \subset \mathcal{E}_h$ the
set of all the edges that do not lie on $\overline{\Gamma_2}$ or $\overline{\Gamma_3}$.
For any $e\in\mathcal{E}_h$, denote by $h_e$ its
length. Related to the triangulation $\mathcal {T}_h$, let
\begin{align*}
&\bSigma:=\left\{\btau\in\left(L^2(\Omega)\right)_s^{2\times2}:
\tau_{ij}|_K\in H^1(K)\ \forall\, K\in\mathcal{T}_h,\ i,j=1,2\right\},\\
&V:=\left\{v\in H^1_{\Gamma_1}(\Omega): v|_K\in H^2(K)\ \forall\,
K\in\mathcal{T}_h\right\}.
\end{align*}
The corresponding finite element spaces are
\begin{align*}
&\bSigma_h:=\left\{\btau_h\in \left(L^2(\Omega)\right)_s^{2\times2}:
\tau_{hij}|_K\in P_l(K)\ \forall\,K\in\mathcal{T}_h,\ i,j=1,2\right\},\\
&V_h:=\left\{v_h\in H^1_{\Gamma_1}(\Omega): v_h|_K\in P_2(K)\
\forall\,K\in\mathcal {T}_h\right\}.
\end{align*}
Here, for a triangle $K\in\mathcal{T}_h$, $P_l(K)$ ($l=0,1$) and $P_2(K)$
are the polynomial spaces on $K$ of degrees $l$ and 2, respectively.
Note that we have the following property
\begin{equation}\label{spaces}
\nabla^2_h V_h\subset\bSigma_h,\quad
\frac{1}{1-\kappa}\bSigma_h-\frac{\kappa}{1-\kappa^2}\left({\rm tr}
\bSigma_h\right)\bI\subset\bSigma_h,
\end{equation}
where $\nabla^2_h V_h|_K:=\nabla^2 (V_h|_K)$ for any
$K\in\mathcal{T}_h$. $\nabla^2_h v$ is defined by the relation
$\nabla^2_h v = \nabla^2 v$ on any element $K\in{\cal T}_h$.

For a function $v\in L^{2}(\Omega)$ with $v|_{K}\in H^{m}(K)$ for
all $K\in\mathcal{T}_{h}$, define the broken norm and seminorm by
\[
\Vert v\Vert_{m,h}=\bigg(\sum_{K\in\mathcal{T}_{h}}\Vert
v\Vert_{m,K}^{2}\bigg)^{1/2},\;\;|v|_{m,h}=\bigg(\sum_{K\in
\mathcal{T}_{h}}|v|_{m,K}^{2}\bigg)^{1/2}.
\]
The above symbols are used in a similar manner when $v$ is a
vector or matrix-valued function. Throughout this paper, $C$
denotes a generic positive constant independent of $h$ and other
parameters, which may take different values at different
occurrences. To avoid writing these constants repeatedly, we use
``$x\lesssim y $" to mean that ``$x\leq C y $". For two vectors
$\bu$ and $\bv$, $\bu\otimes\bv$ is a matrix with $u_iv_j$ as its
$(i,j)$-th component.

Consider two elements $K^+$ and $K^-$ with a common edge $e\in{\cal
E}_h^i$ and let $\bn^+$ and $\bn^-$ be their outward unit normals on
$e$.  For a scalar-valued function $v$, denote its restriction on
$K^{\pm}$ by $v^{\pm}=v|_{K^{\pm}}$. Similarly, for a matrix-valued
function $\btau$, write $\btau^{\pm}=\btau|_{K^{\pm}}$. Then we define
averages and jumps on $e\in \mathcal{E}_h^i$ as follows:
\begin{align*}
\{v\}=\frac{1}{2}(v^++v^-),\quad &
[v]=v^+\bn^++v^-\bn^-,\\
\{\nabla v\}=\frac{1}{2}(\nabla v^++\nabla v^-),
  \quad & [\nabla v]=\nabla
v^+\cdot\bn^++\nabla v^-\cdot\bn^-,\\
\{\btau\}=\frac{1}{2}(\btau^++\btau^-), \quad& [\btau]=\btau^+\bn^+
    +\btau^-\bn^-.
\end{align*}
For $e\in \mathcal{E}_h^b$, the above definitions need to be modified:
\begin{align*}
\{v\}=v, &\quad [v]=v\bnu,\\
\{\nabla v\}=\nabla v, &\quad [\nabla v]
   =\nabla v\cdot\bnu,\\
\{\btau\}=\btau, &\quad [\btau]=\btau\bnu.
\end{align*}
The jump $\llbracket\cdot\rrbracket$ of the vector $\nabla v$ is
\begin{align*}
&\llbracket\nabla v\rrbracket=\frac{1}{2}(\nabla
v^+\otimes\bn^++\bn^  +\otimes\nabla v^++\nabla
v^-\otimes\bn^-+\bn^-\otimes\nabla v^-)
\quad\textrm{on} \ e \in \mathcal{E}_h^i,\\
&\llbracket\nabla v\rrbracket=\frac{1}{2}(\nabla
v\otimes\bnu+\bnu\otimes\nabla v) \quad \textrm{on}\ e \in\mathcal{E}_h^b.
\end{align*}

Define a global lifting operator $\br_0:
\left(L^2(\mathcal{E}^0_h)\right)_s^{2\times2}\rightarrow \bSigma_h$ by
\begin{align}
\int_\Omega\br_0(\bphi):\btau\,dx=-\int_{\mathcal
{E}_h^0}\bphi:\{\btau\}\,ds\quad \forall\,\btau\in
\bSigma_h,\,\bphi\in\left(L^2(\mathcal{E}_h^0)\right)_s^{2\times2}.\label{globallift2}
\end{align}
Moreover, for each $e\in\mathcal{E}_h$, introduce a local lifting
operator
$\br_e:\left(L^2(e)\right)_s^{2\times2}\rightarrow\bSigma_h$ by
\begin{equation}
\label{locallift} \int_\Omega \br_e(\bphi):\btau\,dx
=-\int_{e}\bphi:\{\btau\}\,ds\quad\forall\,\btau\in
\bSigma_h,\,\bphi\in\left(L^2(e)\right)_s^{2\times2}.
\end{equation}
It is easy to check that the following identity holds
\[
\br_0(\bphi)=\sum_{e\in\mathcal {E}_h^0}
\br_e(\bphi|_e)\quad\forall\,\bphi
\in\left(L^2(\mathcal{E}_h^0)\right)_s^{2\times2},
\]
so we have
\begin{equation}\label{lift_relatoin}
\|\br_0(\bphi)\|^2=\|\sum_{e\in
\mathcal{E}_h^0}\br_e(\bphi|_e)\|^2\leq 3\sum_{e\in
\mathcal{E}_h^0}\|\br_e(\bphi|_e)\|^2.
\end{equation}

\subsection{Discontinuous Galerkin formulations}

In \cite{wang13}, a general primal formulation of
CDG methods was presented for a 4th-order elliptic variational inequality of first kind.
The process of deriving CDG schemes for 4th-order elliptic equations can also be
found in \cite{huang10}. Based on the discussions in \cite{wang13} and \cite{huang10},
we introduce five CDG methods for the problem \eqref{EVI} as follows: Find $u_h\in V_{h}$ such that
\begin{equation}\label{dcdg}
B_{h}(u_h,v_h-u_h) + j(v_h) - j(u_h)\ge (f,v_h-u_h)\quad\forall\,v_h\in V_{h},
\end{equation}
where the bilinear form $B_{h}(w,v)=B_{1,h}^{(j)}(w,v)$ with $j=1,\cdots,5$, and
$B_{1,h}^{(j)}(w,v)$ are given next.

The method with $j=1$ is a $C^0$ interior penalty (IP) method, and the bilinear form is
\begin{align}
B_{1,h}^{(1)}(u_h,v_h)=&\int_{\Omega}(1-\kappa)\nabla_h^2
u_h:\nabla_h^2v_h\,dx +\int_{\Omega}\kappa\,{\rm tr}\left(\nabla_h^2
u_h\right)
 {\rm tr}\left(\nabla_h^2v_h\right)dx \nonumber\\
&{}- \int_{\mathcal{E}_h^0}\llbracket \nabla
u_h\rrbracket:\left((1-\kappa)\{\nabla_h^2v_h\}
 +\kappa\,{\rm tr}\left(\{\nabla_h^2v_h\}\right)\bI\right)ds
\notag \nonumber\\
&-\int_{\mathcal{E}_h^0}\llbracket\nabla v_h\rrbracket:
 \left((1-\kappa)\{\nabla_h^2u_h\} +\kappa\,{\rm tr}\left(
  \{\nabla_h^2u_h\}\right)\bI \right)ds\nonumber\\
&{}+\int_{\mathcal{E}_h^0}\eta h^{-1}_e\llbracket\nabla
u_h\rrbracket:\llbracket\nabla v_h\rrbracket \,ds.\label{form1a}
\end{align}
Here $\eta$ is a function, defined to be a constant $\eta_e$ on
each $e\in\mathcal{E}_h^0$, with $\{\eta_e\}_{e\in\mathcal{E}_h^0}$
having a uniform positive bound from above and below.
%In
%(\ref{eq:trans}), let $v_h = w_h -u_h$ with $w_h\in V_h$, we
%obtain
%\begin{align*}
%B_{1,h}^{(1)}(u_h,w_h-u_h) + \int_{\Gamma_3} g\ \lambda_h\ w_h\, ds - \int_{\Gamma_3} g\ \lambda_h\ u_h\, ds = \int_{\Omega}f(w_h-u_h)\,dx
%\quad\quad \forall\, w_h\in V_h.
%\end{align*}
%And due to \eqref{lar_h}, we have
%\begin{align}\label{discrete1}
%B_{1,h}^{(1)}(u_h,w_h-u_h) + j(w_h) - j(u_h) \geq \int_{\Omega}f(w_h-u_h)\,dx
%\quad\quad \forall\, w_h\in V_h.
%\end{align}
For a compact formulation, we can use lifting operator $\br_0$ (cf.\ \eqref{globallift2})
to get
\begin{align}
B_{2,h}^{(1)}(u_h,v_h)=&\int_{\Omega}(1-\kappa)\nabla_h^2
u_h:\left(\nabla_h^2 v_h
 +\br_0(\llbracket\nabla v_h\rrbracket)\right)dx\nonumber\\
&{}+\int_{\Omega}\kappa\,{\rm tr}\left(\nabla_h^2u_h\right)
 {\rm tr}\left(\nabla_h^2 v_h
+\br_0(\llbracket\nabla v_h\rrbracket)\right)dx \nonumber\\
&{}+ \int_{\Omega}\br_0\left(\llbracket \nabla
u_h\rrbracket\right):\left((1-\kappa)\nabla_h^2v_h
 +\kappa\,{\rm
 tr}\left(\nabla_h^2v_h\right)\bI\right)dx\nonumber\\
&{} +\int_{\mathcal{E}_h^0}\eta h^{-1}_e\llbracket\nabla
u_h\rrbracket:\llbracket\nabla v_h\rrbracket\,ds.\label{form1b}
\end{align}
A similar $C^0$ IP method was studied in \cite{brenner05}.

The two formulas (\ref{form1a}) and
(\ref{form1b}) are equivalent on the finite element spaces $V_h$,
so either form can be used to compute the  finite element solution
$u_h$. In this paper, we give a priori error estimates strictly
based on the first formula $B_{1,h}^{(1)}$. Because of the
equivalence of these two formulations on $V_h$, we will prove the
stability for the second formula $B_{2,h}^{(1)}$ on $V_h$, which
ensures the stability for the first formulation $B_{1,h}^{(1)}$ on
$V_h$. This comment is valid for the other CDG methods introduced next.

%We mow introduce four more CDG methods for the variational inequality
%(\ref{EVI}).  The methods are all of the form \eqref{discrete1}, and so
%we will only list the corresponding bilinear form.

Motivated by related DG methods for the second order elliptic
problem, we can define the $C^0$ non-symmetric interior penalty
(NIPG) formulation,
\begin{align}
B_{1,h}^{(2)}(u_h,v_h)=&\int_{\Omega}(1-\kappa)\nabla_h^2
u_h:\nabla_h^2v_h\,dx +\int_{\Omega}\kappa\,{\rm tr}\left(\nabla_h^2
u_h\right)
 {\rm tr}\left(\nabla_h^2v_h\right)dx \nonumber\\
&{}+ \int_{\mathcal{E}_h^0}\llbracket \nabla
u_h\rrbracket:\left((1-\kappa)\{\nabla_h^2v_h\}
 +\kappa\,{\rm tr}\left(\{\nabla_h^2v_h\}\right)\bI\right)ds
\notag \nonumber\\
&-\int_{\mathcal{E}_h^0}\llbracket\nabla v_h\rrbracket:
\left((1-\kappa)\{\nabla_h^2u_h\}
 +\kappa\,{\rm tr}\left(\{\nabla_h^2u_h\}\right)\bI \right)ds\nonumber\\
&{}+\int_{\mathcal{E}_h^0}\eta h^{-1}_e\llbracket\nabla
u_h\rrbracket:\llbracket\nabla v_h\rrbracket \, ds,\label{form2a}
\end{align}
or equivalently,
\begin{align}
B_{2,h}^{(2)}(u_h,v_h) =&\int_{\Omega}(1-\kappa)\nabla_h^2
u_h:\left(\nabla_h^2 v_h
+\br_0(\llbracket\nabla v_h\rrbracket)\right)dx\nonumber\\
&{}+\int_{\Omega}\kappa\,{\rm tr}\left(\nabla_h^2u_h\right)
 {\rm tr}\left(\nabla_h^2 v_h
+\br_0(\llbracket\nabla v_h\rrbracket)\right)dx \nonumber\\
&{}- \int_{\Omega}\br_0\left(\llbracket\nabla
u_h\rrbracket\right):\left((1-\kappa)\nabla_h^2v_h
 +\kappa\,{\rm tr}\left(\nabla_h^2v_h\right)\bI\right)dx\nonumber\\
&{} +\int_{\mathcal{E}_h^0}\eta h^{-1}_e\llbracket\nabla
u_h\rrbracket:\llbracket\nabla v_h\rrbracket \, ds.\label{form2b}
\end{align}

%Using the local lifting operator $\br_e$, we can give the third
%example. Taking
%\[
%\left\{
%\begin{aligned}
%\widehat{\nabla u}_h =& \left\{\nabla
%u_h\right\} \quad {\rm on}\; e\in\mathcal{E}_h, \quad
%\widehat{\nabla u}_h = {\bf 0} \quad {\rm on}\; e\in\Gamma_1, \\
%\widehat{\bsigma}_h =& -(1-\kappa)\{\nabla_h^2u_h\} -\kappa\,{\rm
%tr}(\{\nabla_h^2u_h\})\bI -(1-\kappa)\{\br_i(\llbracket\nabla
%u_h\rrbracket)\} -\kappa\{{\rm tr}
%(\br_i(\llbracket\nabla u_h\rrbracket))\bI\} \\
%&-(1-\kappa)\{\eta\, \br_e(\llbracket\nabla u_h\rrbracket)\}
%-\kappa\{\eta\,{\rm tr}(\br_e(\llbracket\nabla
%u_h\rrbracket))\}\bI\quad {\rm on}\; e\in\mathcal{E}_h^0,\\
%\widehat{\bsigma}_{h}\bnu =&\ 0 \quad {\rm on}\; e\in\Gamma_2\cup\Gamma_3,
%\end{aligned}
%\right.
%\]
%from \eqref{eq:primal}, we get

The CDG method with $j=3$ has the bilinear form
\begin{equation}
\begin{split}
B_{1,h}^{(3)}(u_h,v_h)=&\int_{\Omega}(1-\kappa)\nabla_h^2
u_h:\nabla_h^2v_h\,dx +\int_{\Omega}\kappa\,{\rm tr}\left(\nabla_h^2
u_h\right)
 {\rm tr}\left(\nabla_h^2v_h\right)dx \\
&{}- \int_{\mathcal{E}_h^0}\llbracket \nabla
u_h\rrbracket:\left((1-\kappa)\{\nabla_h^2v_h\}
 +\kappa\,{\rm tr}\left(\{\nabla_h^2v_h\}\right)\bI\right)ds
\notag \\
&-\int_{\mathcal{E}_h^0}\llbracket\nabla v_h\rrbracket:
\left((1-\kappa)\{\nabla_h^2u_h\}
 +\kappa\,{\rm tr}\left(\{\nabla_h^2u_h\}\right)\bI\right)ds\\
  &+\int_{\Omega}\br_0(\llbracket\nabla v_h\rrbracket):
  \big((1-\kappa)\br_0(\llbracket\nabla u_h\rrbracket)
 +\kappa\,{\rm tr}
(\br_0(\llbracket\nabla u_h\rrbracket))\bI\big)\,dx\\
&+\sum_{e\in \mathcal{E}_h^0}\int_{\Omega}\eta
\left((1-\kappa)\br_e(\llbracket\nabla
u_h\rrbracket):\br_e(\llbracket\nabla v_h\rrbracket)+\kappa\, {\rm
tr}(\br_e(\llbracket\nabla u_h\rrbracket)) {\rm
tr}(\br_e(\llbracket\nabla v_h\rrbracket))\right)dx,
\end{split}\label{form3a}
\end{equation}
or equivalently,
\begin{equation}
\begin{split}
B_{2,h}^{(3)}(u_h,v_h)=&\int_{\Omega}(1-\kappa)\left(\nabla_h^2 u_h
+\br_0(\llbracket\nabla u_h\rrbracket)\right):\left(\nabla_h^2 v_h
+\br_0(\llbracket\nabla v_h\rrbracket)\right)dx \\
& +\int_{\Omega}\kappa\, {\rm tr}\left(\nabla_h^2 u_h
 +\br_0(\llbracket\nabla u_h\rrbracket)\right)
{\rm tr}\left(\nabla_h^2 v_h+\br_0(\llbracket\nabla
v_h\rrbracket)\right)dx  \\
&+\sum_{e\in \mathcal{E}_h^0}\int_{\Omega}\eta
\left((1-\kappa)\br_e(\llbracket\nabla
u_h\rrbracket):\br_e(\llbracket\nabla v_h\rrbracket)+\kappa\, {\rm
tr}(\br_e(\llbracket\nabla u_h\rrbracket)) {\rm
tr}(\br_e(\llbracket\nabla v_h\rrbracket))\right)dx,
\end{split}
\label{form3b}
\end{equation}
which is the CDG formulation proposed in \cite{wells07}.

%With the choice of
%\[
%\left\{
%\begin{aligned}
%\widehat{\nabla u}_h =& \left\{\nabla
%u_h\right\} \quad {\rm on}\ e\in\mathcal{E}_h, \quad
%\widehat{\nabla u}_h = {\bf 0} \quad {\rm on}\; e\in\Gamma_1,\\
%\widehat{\bsigma}_h =& -(1-\kappa)\{\nabla_h^2u_h\} -\kappa\,{\rm
%tr}(\{\nabla_h^2u_h\})\bI -(1-\kappa)\{\eta \,
%\br_e(\llbracket\nabla u_h\rrbracket)\}\\
%& -\kappa\{\eta \,{\rm tr}(\br_e(\llbracket\nabla
%u_h\rrbracket))\}\bI \quad {\rm on}\ e\in\mathcal{E}_h^0,\\
%\widehat{\bsigma}_{h}\bnu =&\ 0 \quad {\rm on}\; e\in\Gamma_2\cup\Gamma_3,
%\end{aligned}
%\right.
%\]
%we obtain

The bilinear form of the CDG scheme with $j=4$ is
\begin{equation}
\begin{split}
B_{1,h}^{(4)}(u_h,v_h)=&\int_{\Omega}(1-\kappa)\nabla_h^2
u_h:\nabla_h^2v_h\,dx +\int_{\Omega}\kappa\,{\rm tr}\left(\nabla_h^2 u_h\right)
 {\rm tr}\left(\nabla_h^2v_h\right)dx \\
&{}- \int_{\mathcal{E}_h^0}\llbracket \nabla
u_h\rrbracket:\left((1-\kappa)\{\nabla_h^2v_h\}
 +\kappa\,{\rm tr}\left(\{\nabla_h^2v_h\}\right)\bI\right)ds\\
&-\int_{\mathcal{E}_h^0}\llbracket\nabla v_h\rrbracket:
\left((1-\kappa)\{\nabla_h^2u_h\}
 +\kappa\,{\rm tr}\left(\{\nabla_h^2u_h\}\right)\bI \right)ds\\
&+\sum_{e\in \mathcal{E}_h^0}\int_{\Omega}\eta
\left((1-\kappa)\br_e(\llbracket\nabla
u_h\rrbracket):\br_e(\llbracket\nabla v_h\rrbracket)+\kappa\, {\rm
tr}(\br_e(\llbracket\nabla u_h\rrbracket)) {\rm
tr}(\br_e(\llbracket\nabla v_h\rrbracket))\right)dx,
\end{split}
\label{form4a}
\end{equation}
or equivalently,
\begin{equation}
\begin{split}
B_{2,h}^{(4)}(u_h,v_h)=&\int_{\Omega}(1-\kappa)\nabla_h^2 u_h
:\left(\nabla_h^2 v_h +\br_0(\llbracket\nabla v_h\rrbracket)\right)dx\\
&{} +\int_{\Omega}\kappa\, {\rm
tr}\left(\nabla_h^2 u_h \right) {\rm tr}\left(\nabla_h^2
v_h+\br_0(\llbracket\nabla v_h\rrbracket)\right)dx  \\
&{}+ \int_{\Omega}\br_0\left(\llbracket \nabla
u_h\rrbracket\right):\left((1-\kappa)\nabla_h^2v_h
 +\kappa\,{\rm tr}\left(\nabla_h^2v_h\right)\bI\right)dx\\
&+\sum_{e\in \mathcal{E}_h^0}\int_{\Omega}\eta
\left((1-\kappa)\br_e(\llbracket\nabla
u_h\rrbracket):\br_e(\llbracket\nabla v_h\rrbracket)+\kappa\, {\rm
tr}(\br_e(\llbracket\nabla u_h\rrbracket)) {\rm
tr}(\br_e(\llbracket\nabla v_h\rrbracket))\right)dx,
\end{split}\label{form4b}
\end{equation}
which is the CDG formulation extended from the DG method of
\cite{bassi97} for elliptic problems of second order.

%Choosing
%\[
%\left\{
%\begin{aligned}
%\widehat{\nabla u}_h=& \left\{\nabla
%u_h\right\} \quad {\rm on}\ e\in\mathcal{E}_h, \quad
%\widehat{\nabla u}_h = {\bf 0} \quad {\rm on}\; e\in\Gamma_1,\\
%\widehat{\bsigma}_h= &-(1-\kappa)\{\nabla_h^2u_h\} -\kappa\,{\rm
%tr}(\{\nabla_h^2u_h\})\bI -(1-\kappa)\{\br_i(\llbracket\nabla
%u_h\rrbracket)\} -\kappa\{{\rm tr}
%(\br_i(\llbracket\nabla u_h\rrbracket))\bI\} \\
%& +\eta h_e^{-1}\llbracket\nabla
%u_h\rrbracket \quad {\rm on}\ e\in\mathcal{E}_h^0, \\
%\widehat{\bsigma}_{h}\bnu =&\ 0 \quad {\rm on}\; e\in\Gamma_2\cup\Gamma_3,
%\end{aligned}
%\right.
%\]

For the LCDG method (\cite{huang10}), the bilinear form is
\begin{align}
B_{1,h}^{(5)}(u_h,v_h):=&\int_{\Omega}(1-\kappa)\nabla_h^2
u_h:\nabla_h^2v_h\,dx +\int_{\Omega}\kappa\,{\rm tr}\left(\nabla_h^2
u_h\right)
 {\rm tr}\left(\nabla_h^2v_h\right)dx \nonumber\\
&{}- \int_{\mathcal{E}_h^0}\llbracket \nabla
u_h\rrbracket:\left((1-\kappa)\{\nabla_h^2v_h\}
 +\kappa\,{\rm tr}\left(\{\nabla_h^2v_h\}\right)\bI\right)ds
\notag \nonumber\\
&-\int_{\mathcal{E}_h^0}\llbracket\nabla v_h\rrbracket:
\left((1-\kappa)\{\nabla_h^2u_h\}
 +\kappa\,{\rm tr}\left(\{\nabla_h^2u_h\}\right)\bI\right)ds\nonumber\\
&+\int_{\Omega}\br_0(\llbracket\nabla v_h\rrbracket):
\big((1-\kappa)\br_0(\llbracket\nabla u_h\rrbracket)
 +\kappa\,{\rm tr}(\br_0(\llbracket\nabla u_h\rrbracket))
 \bI\big)\,dx\nonumber\\
&{} +\int_{\mathcal{E}_h^0}\eta h^{-1}_e\llbracket\nabla
  u_h\rrbracket:\llbracket\nabla v_h\rrbracket \, ds,\label{form5a}
\end{align}
or equivalently,
\begin{align}
B_{2,h}^{(5)}(u_h,v_h):=&\int_{\Omega}(1-\kappa)\left(\nabla_h^2u_h
+\br_0(\llbracket\nabla u_h\rrbracket)\right): \left(\nabla_h^2v_h
+\br_0(\llbracket\nabla v_h\rrbracket)\right)dx \nonumber\\
&{}+\int_{\Omega}\kappa\,{\rm tr}\left(\nabla_h^2u_h
+\br_0(\llbracket\nabla u_h\rrbracket)\right) {\rm
tr}\left(\nabla_h^2v_h+\br_0(\llbracket
  \nabla v_h\rrbracket)\right)dx \nonumber\\
&{} +\int_{\mathcal{E}_h^0}\eta h^{-1}_e\llbracket\nabla
u_h\rrbracket:\llbracket\nabla v_h\rrbracket \, ds.\label{form5b}
\end{align}

\section{Consistency, boundedness and stability}\label{sec:3}
\setcounter{equation}0

First, we address the consistency of the methods (\ref{dcdg}).

\begin{lemma}\label{consisency}
For the solution of the problem $(\ref{EVI})$, assume $u\in H^3(\Omega)$.
Then for all the five CDG methods with
$B_{h}(w,v)=B_{1,h}^{(j)}(w,v)$, $1\le j \le 5$, we have
\[B_{h}(u,v_h-u)\ge (f,v_h-u)\quad\forall\,v_h\in V_{h}.\]
\end{lemma}
\begin{proof}
Noting $\llbracket\nabla u\rrbracket =0$ on each edge $e \in {\cal
E}_h^i$, we use (\ref{moment}) to get
\begin{align*}
B_{h}(u,v_h-u) =&\int_{\Omega}(1-\kappa)\nabla^{2}u:\nabla_h^{2}
(v_h-u)\, dx +\int_{\Omega}\kappa\,{\rm tr}\left(\nabla^{2}u\right)
{\rm
tr}\left(\nabla_h^{2}(v_h-u)\right) dx\\
{} &-\int_{\mathcal{E}_h^0}\llbracket\nabla (v_h-u)\rrbracket:
\left((1-\kappa)\nabla^2u
 +\kappa\,{\rm tr}\left(\nabla^2u\right)\bI
 \right)ds\\
=&-\sum_{K\in\mathcal{T}_{h}}\int_{K}\bsigma:\nabla_h^{2}
(v_h-u)\,dx+\int_{\mathcal{E}_{h}^0}\llbracket\nabla
(v_h-u)\rrbracket:\bsigma\,ds.
\end{align*}
Using Lemma \ref{identities} and noticing $[\bsigma]={\bf 0}$ on each edge
$e \in {\cal E}_h^i$, we have
\begin{align*}
\sum_{K\in\mathcal{T}_{h}}\int_{K}\bsigma:\nabla_h^{2}(v_h-u)\,dx
&=-\sum_{K\in\mathcal{T}_{h}}\int_{K}\nabla
(v_h-u)\cdot(\nabla\cdot \bsigma)\,dx \\
&{}\quad +\sum_{K\in\mathcal{T}_{h}}\int_{\partial K}\nabla
(v_h-u)\cdot(\bsigma \bn_K)\,ds\\
&=-\int_{\Omega}\nabla(v_h-u)\cdot(\nabla\cdot\bsigma)\,dx
+\int_{\mathcal{E}_{h}}\llbracket\nabla(v_h-u)\rrbracket:\bsigma\,ds.
\end{align*}
Combining the above two equations, we obtain
\begin{align*}
B_{h}(u,v_h-u)=&\int_{\Omega}\nabla
(v_h-u)\cdot(\nabla\cdot\bsigma)\,dx -\int_{\Gamma_2\cup\Gamma_3}\llbracket\nabla
(v_h-u)\rrbracket:\bsigma\,ds\\
=& \int_{\Omega}\nabla(v_h-u)\cdot(\nabla\cdot\bsigma)\,dx -\int_{\Gamma_2\cup \Gamma_3} \sigma_\tau\ \partial_\tau (v_h - u)\,ds\\
 =&{} - \int_\Omega \nabla\cdot(\nabla\cdot\bsigma) (v_h-u)\,dx
 + \langle (\nabla\cdot\bsigma)\cdot\bnu,v_h-u \rangle_{1/2,\Gamma}\\
&{}  -\int_{\Gamma_2\cup \Gamma_3} \sigma_\tau\, \partial_\tau (v_h - u)\,ds.
\end{align*}
Here, the second equation comes from the relation \eqref{relation4}, and
the last equation holds by \eqref{relation2}.

We apply the relation \eqref{relation5}, Lemma \ref{identities},
\eqref{Lar2} and \eqref{relation1} to obtain
\begin{align*}
B_{h}(u,v_h-u)=&
-\int_{\Omega}\nabla\cdot(\nabla\cdot\bsigma)(v_h-u)\,dx
-\int_{\Gamma_3} g\ \lambda\ v_h\,ds + \int_{\Gamma_3} g\ \lambda\ u\,ds\\
=& \int_\Omega f(v_h-u)\,dx -\int_{\Gamma_3} g\ \lambda\ v_h\,ds + \int_{\Gamma_3} g\ |u|\,ds\\
\ge& \int_\Omega f(v_h-u)\,dx -\int_{\Gamma_3} g\ |v_h|\,ds + \int_{\Gamma_3} g\ |u|\,ds.
\end{align*}
So the stated result holds.
\end{proof}

Let $V(h):=V_h+V\cap H^3(\Omega)$ and define two
mesh-dependent energy norms by
\[ |v|_*^2:=|v|^2_{2,h}+\sum_{e\in\mathcal{E}_h^0}h^{-1}_e\|\llbracket
\nabla v\rrbracket\|^2_{0,e}, \quad \interleave
v\interleave^2:=|v|_*^2+\sum_{K\in\mathcal{T}_h}h^2_K|v|_{3,K}^2,\quad v\in V(h). \]
To show these formulas define norms, we only need
prove that $|v|_*=0$ and $v\in V(h)$ imply $v=0$.  From
$|v|_{2,h}=0$, we have $v|_K\in P_1(K)$ and so $\nabla v$ is
piecewise constant. Let $e$ be the common edge of two neighboring elements $K^{+}$
and $K^{-}$. From $\|\llbracket \nabla v\rrbracket\|_{0,e}=0$, we
obtain $(\nabla v)^+=(\nabla v)^-$. Thus, $\nabla v$ is constant in
$\Omega$ and so $v\in P_1(\Omega)$. Since $v=0$ on $\Gamma_1$, we
conclude that $v=0$ in $\Omega$.

Before presenting boundedness and stability results of the bilinear
forms, we give a useful estimate for the lifting operator $\br_e$.

\begin{lemma}\label{lem:jump}
There exist two positive constants $C_1\le C_2$ such that for any $v\in V(h)$
and $e\in\mathcal{E}_h^0$,
\begin{equation}\label{eq:jump}
C_1 h_e^{-1}\|\llbracket\nabla
v\rrbracket\|^2_{0,e}\leq\|\br_e(\llbracket\nabla
v\rrbracket)\|^2_{0,h}\leq C_2 h_e^{-1}\|\llbracket\nabla
v\rrbracket\|^2_{0,e}.
\end{equation}
\end{lemma}
\begin{proof}
The second inequality was proved in \cite{huang10}. For $v\in
V\cap H^3(\Omega)$, $\llbracket\nabla v\rrbracket=0$ on $e\in {\cal
E}_h^0$.  So we only need to consider the case $v\in V_h$. By the
formula between (4.4) and (4.5) in \cite{arnold02}, we know
\begin{equation}
h^{-1}_e\|\bvarphi\|^2_{0,e}\lesssim
\|r^*_e(\bvarphi)\|^2_{0,\Omega}\lesssim
h^{-1}_e\|\bvarphi\|^2_{0,e}\quad \forall\,\bvarphi\in [P_1(e)]^2,
\end{equation}
where the lifting operator $r^*_e: (L^2(e))^2 \rightarrow W_h$ is defined by
\[\int_\Omega r^*_e(\bv)\cdot \bw_h\, dx = - \int_e \bv\cdot \{\bw_h\}\, ds, 
\quad \forall\, \bw_h\in W_h.\]
Here, $W_h:=\left\{\bw_h\in \left(L^2(\Omega)\right)^{2}:
\bw_{hi}|_K\in P_l(K),\ \forall\,K\in\mathcal{T}_h,i=1,2\right\}$.

For two matrix-valued functions $\bphi=(\phi_{ij})_{2\times 2}$ and
$\btau=(\tau_{ij})_{2\times 2}$, let
$\bphi_1=(\phi_{11},\phi_{21})^t$, $\bphi_2=(\phi_{12},\phi_{22})^t$,
$\btau_1=(\tau_{11},\tau_{21})^t$, $\btau_2=(\tau_{12},\tau_{22})^t$,
so that $\bphi=(\bphi_1,\bphi_2)$, $\btau=(\btau_1,\btau_2)$. Then
\begin{align*}
\int_\Omega \br_e(\bphi):\btau\,dx&=-\int_e\bphi:\{\btau\}\,ds
=-\int_e\bphi_1\cdot\{\btau_1\}\,ds-\int_e\bphi_2\cdot\{\btau_2\}\,ds\\
&=\int_\Omega r^*_e(\bphi_1)\cdot\btau_1\,dx+\int_\Omega
r^*_e(\bphi_2)\cdot\btau_2\,dx=\int_\Omega
(r^*_e(\bphi_1),r^*_e(\bphi_2)):\btau\,dx,
\end{align*}
for all $\btau\in \bsigma_h$. So $\br_e(\bphi)=(r^*_e(\bphi_1),r^*_e(\bphi_2))$,
$\|\br_e(\bphi)\|_{0,\Omega}^2=\|r^*_e(\bphi_1)\|_{0,\Omega}^2
+\|r^*_e(\bphi_2)\|_{0,\Omega}^2$, and
\begin{align*}
h^{-1}_e\|\bphi\|^2_{0,e}&=h^{-1}_e(\|\bphi_1\|^2_{0,e}+\|\bphi_2\|^2_{0,e})
\lesssim \|r^*_e(\bphi_1)\|_{0,\Omega}^2+\|r^*_e(\bphi_2)\|_{0,\Omega}^2
=\|\br_e(\bphi)\|_{0,\Omega}^2.
\end{align*}
Let $\bphi=\llbracket\nabla v\rrbracket$, then the first inequality follows.
\end{proof}

From (\ref{eq:jump}) and (\ref{lift_relatoin}), we have
\[\|\br_0(\llbracket\nabla
v\rrbracket)\|^2_{0,h}=
\|\sum_{e\in\mathcal{E}_h^0}\br_e(\llbracket\nabla
v\rrbracket)\|^2_{0,h}\leq 3 C_2
\sum_{e\in\mathcal{E}_h^0}h_e^{-1}\|\llbracket\nabla
v\rrbracket\|^2_{0,e}.\]

For the boundedness of the primal forms $B^{(j)}_{1,h}$ with
$j=1,\cdots,5$, first notice that $\|{\rm tr}(\btau)\|_{0,h}
\lesssim \|\btau\|_{0,h}$. By the Cauchy-Schwarz inequality and
Lemma \ref{lem:jump}, we get the following inequalities:
\begin{align}
&\int_\Omega\nabla^{2}_h w:\nabla_h^{2}v\,dx
\leq |w|_{2,h}|v|_{2,h},\label{1}\\
&\int_\Omega\br_0(\llbracket\nabla
w\rrbracket):\br_0(\llbracket\nabla v\rrbracket)\,dx \lesssim
\left(\sum_{e\in\mathcal{E}_h^0}h_e^{-1}\|\llbracket\nabla
w\rrbracket\|^2_{0,e}\right)^{1/2}
\left(\sum_{e\in\mathcal{E}_h^0}h_e^{-1}\|\llbracket\nabla
v\rrbracket\|^2_{0,e}\right)^{1/2},\label{2}\\
&\int_{\mathcal{E}_h^0}\eta h_e^{-1}\llbracket\nabla
w\rrbracket:\llbracket\nabla v\rrbracket\,ds \leq \sup_{e\in
\mathcal{E}_h^0} \eta_e
\left(\sum_{e\in\mathcal{E}_h^0}h_e^{-1}\|\llbracket\nabla
w\rrbracket\|^2_{0,e}\right)^{1/2}
\left(\sum_{e\in\mathcal{E}_h^0}h_e^{-1}\|\llbracket\nabla
v\rrbracket\|^2_{0,e}\right)^{1/2},\label{3}\\
&\sum_{e\in \mathcal{E}_h^0}\int_{\Omega}\eta\,
\br_e(\llbracket\nabla w\rrbracket):\br_e(\llbracket\nabla
v\rrbracket)\,dx \lesssim \sup_{e\in \mathcal{E}_h^0} \eta_e
\left(\sum_{e\in\mathcal{E}_h^0}h_e^{-1}\|\llbracket\nabla
w\rrbracket\|^2_{0,e}\right)^{1/2}
\left(\sum_{e\in\mathcal{E}_h^0}h_e^{-1}\|\llbracket\nabla
v\rrbracket\|^2_{0,e}\right)^{1/2}.\label{4}
\end{align}
Using the trace inequality $\|\nabla^2 v\|_{0,e}\lesssim h_e^{-1}|v|^2_{2,K}
+h_e|v|^2_{3,K}$ with $e$ an edge of $K$, we have
\begin{align}
\int_{\mathcal{E}_h^0}\llbracket \nabla
w\rrbracket:\{\nabla_h^2v\}\,ds & =\sum_{e\in{\cal
E}_h^0}\int_{e}\llbracket \nabla
w\rrbracket:\{\nabla_h^2v\}\,ds \nonumber\\
& \leq \left(\sum_{e\in{\cal E}_h^0}h_e^{-1}\|\llbracket\nabla w
 \rrbracket\|^2_{0,e}\right)^{1/2}\left(\sum_{e\in{\cal E}_h^0}h_e\|
 \{\nabla_h^2v\}\|^2_{0,e}\right)^{1/2}\nonumber\\
&\lesssim \left(\sum_{e\in{\cal E}_h^0}h_e^{-1}\|\llbracket \nabla
w\rrbracket\|^2_{0,e}\right)^{1/2}
\left(\sum_{K\in\mathcal{T}_h}(|v|_{2,K}^2+h^2_K|v|_{3,K}^2)\right)^{1/2}.\label{5}
\end{align}
The inequalities (\ref{1}) and (\ref{5}) are needed by all
bilinear forms. For the CDG methods with the bilinear form
$B_{1,h}^{(j)}$, $j=1, 2, 5$, the inequality (\ref{3}) is needed. The
inequality (\ref{2}) is needed by the formulas $B_{1,h}^{(j)}$ with
$j=3, 5$. The methods with the bilinear forms $B_{1,h}^{(j)}$,
$j=3, 4$, need the inequality (\ref{4}). Then we have the following result.

\begin{lemma}[Boundedness]
Let $B_{h}=B_{1,h}^{(j)}$ with $j=1,\cdots,5$. Then
\begin{align}
B_{h}(w,v)\lesssim &\interleave w\interleave \interleave
v\interleave \quad\quad\quad \forall\,(w,v)\in V(h)\times
V(h).\label{eq:boundedness}
\end{align}
\end{lemma}

For stability over $V_h$, note that $\interleave v\interleave=|v|_*$ for any
$v\in V_h$. Formulations $B_{1,h}^{(j)}$
and $B_{2,h}^{(j)}$ are equivalent on $V_h$, so we just need to
prove the stability for $B_{2,h}^{(j)}$ based on $|\cdot|_*$. We
use the Cauchy-Schwarz inequality and Lemma \ref{lem:jump} to get
\begin{align*}
B_{2,h}^{(1)}(v,v)=&(1-\kappa)\int_{\Omega}\nabla_h^2 v:\nabla_h^2 v
\,dx + \kappa \int_{\Omega} \left({\rm tr}(\nabla_h^2 v)\right)^2 dx
+ 2(1-\kappa)\int_{\Omega} \nabla_h^2 v:
\br_0(\llbracket\nabla v\rrbracket)\,dx\\
&{}+ 2\kappa \int_{\Omega} {\rm tr}\left(\nabla_h^2 v\right){\rm
tr}\left(\br_0(\llbracket\nabla v\rrbracket)\right) dx
+\int_{\mathcal{E}_h^0}\eta
h^{-1}_e|\llbracket\nabla v\rrbracket|^2 ds\\
\geq & (1-\kappa)|v|_{2,h}^2 + \kappa\|\Delta_h
v\|_{0,h}^2-(1-\kappa)\left(\epsilon|v|_{2,h}^2+\frac{1}{\epsilon}
\|\br_0(\llbracket\nabla v\rrbracket)\|_{0,h}^2\right)\\
&-\kappa\left(\|\Delta_h v\|_{0,h}^2+\|{\rm tr}\left(\br_0
(\llbracket \nabla v\rrbracket)\right)\|_{0,h}^2\right)+\eta_0
\sum_{e\in\mathcal{E}_h^0}h^{-1}_e\|\llbracket \nabla v\rrbracket\|^2_{0,e}\\
\geq&(1-\epsilon)(1-\kappa)|v|_{2,h}^2+\left(\eta_0-\frac{3(1-\kappa)C_2}{\epsilon}
-6C_2\kappa\right)\sum_{e\in\mathcal{E}_h^0}h^{-1}_e\|\llbracket\nabla
v \rrbracket\|^2_{0,e},
\end{align*}
where $0<\epsilon<1$ is a constant and $C_2$ is the generic positive
constant in (\ref{eq:jump}). Therefore, stability is valid for the $C^0$
IP method when
\[ \min_{e\in\mathcal{E}_h^0}\eta_e=\eta_0>3(1-\kappa)C_2+6C_2\kappa=3(1+\kappa)C_2. \]
Next,
\begin{align*}
B_{2,h}^{(2)}(v,v)=&\int_{\Omega}(1-\kappa)\nabla_h^2 v:\nabla_h^2 v
\,dx +\int_{\Omega}\kappa\,\left({\rm tr}(\nabla_h^2 v)\right)^2 dx
+\int_{\mathcal{E}_h^0}\eta h^{-1}_e\left(\llbracket\nabla
v\rrbracket\right)^2 ds\\
\geq&(1-\kappa)|v|_{2,h}^2+\eta_0\sum_{e\in\mathcal{E}_h^0}h^{-1}_e
 \|\llbracket \nabla v\rrbracket\|^2_{0,e}.
\end{align*}
So stability is valid for the $C^0$ NIPG method for any $\eta_0>0$.
This property is the reason why the method with the bilinear form
$B_{2,h}^{(2)}$ is useful even though $B_{2,h}^{(2)}$ is not
symmetric.
\begin{align*}
B_{2,h}^{(4)}(v,v)\geq&(1-\kappa)|v|_{2,h}^2 + \kappa \|\Delta_h
v\|_{0,h}^2 + 2(1-\kappa)\int_{\Omega} \nabla_h^2 v:
\br_0(\llbracket\nabla v\rrbracket)\,dx\\
&{}+ 2\kappa \int_{\Omega} \Delta_h v\,{\rm
tr}\left(\br_0(\llbracket\nabla v\rrbracket)\right)dx
+\eta_0\sum_{e\in\mathcal{E}_h^0}\left((1-\kappa)\|\br_e(\llbracket\nabla
v\rrbracket)\|^2_{0,h}+\kappa\|{\rm tr}(\br_e(\llbracket\nabla
v\rrbracket))\|^2_{0,h}\right)\\
\geq & (1-\kappa)|v|_{2,h}^2 + \kappa\|\Delta_h
v\|_{0,h}^2-(1-\kappa)\left(\epsilon|v|_{2,h}^2+\frac{1}{\epsilon}\|\br_0
(\llbracket\nabla v\rrbracket)\|_{0,h}^2\right)-\kappa\|\Delta_h v\|_{0,h}^2\\
&-\kappa\|{\rm tr}\left(\br_0(\llbracket\nabla v\rrbracket)
\right)\|_{0,h}^2+\eta_0C_1(1-\kappa)
\sum_{e\in\mathcal{E}_h^0}h^{-1}_e\|\llbracket \nabla
v\rrbracket\|^2_{0,e}+\eta_0\kappa \sum_{e\in\mathcal{E}_h^0}\|{\rm
tr}(\br_e(\llbracket\nabla v\rrbracket))\|^2_{0,h}\\
\geq&
(1-\epsilon)(1-\kappa)|v|_{2,h}^2+(1-\kappa)\left(\eta_0C_1-\frac{3C_2}{\epsilon}\right)
\sum_{e\in\mathcal{E}_h^0}h^{-1}_e\|\llbracket\nabla v\rrbracket\|^2_{0,e}\\
&{}+(\eta_0\kappa-3\kappa) \sum_{e\in\mathcal{E}_h^0}\|{\rm
tr}(\br_e(\llbracket\nabla v\rrbracket))\|^2_{0,h}.
\end{align*}
Since $C_2>C_1$, $\eta_0>3$ is guaranteed from $\eta_0>3C_2/C_1$.
Thus, stability is valid for this CDG formulation when
$\eta_0>3C_2/C_1$. For the method of Wells-Dung corresponding to the
form $B_{2,h}^{(3)}$ and the LCDG method corresponding to the form
$B_{2,h}^{(5)}$, stability can be analyzed by a similar argument
(cf.\ \cite{wells07} and \cite{huang10}, respectively), with
$\eta_0>0$.

Summarizing, we have shown the following result.

\begin{lemma}[Stability]\label{lem:sta}
Let $B_{h}=B_{2,h}^{(j)}$ with $j=1,\cdots,5$. Assume
\[ \min_{e\in{\cal E}^0_h}\eta_e>3\left(1+\kappa\right)C_2\ {\rm for\ }j=1 \]
and
\[ \min_{e\in{\cal E}^0_h}\eta_e>3\,C_2/C_1\ {\rm for\ }j=4, \]
with $C_1$ and $C_2$ the constants in the inequality \eqref{eq:jump}.  Then,
\begin{equation}
\interleave v\interleave^2 \lesssim  B_{h}(v,v)
\quad\forall\, v\in V_h. \label{eq:stability}
\end{equation}
\end{lemma}

We further conclude that the stability is also valid for $B_{1,h}^{(j)}$ with $j=1,\cdots,5$.

\section{Error analysis}\label{sec:4}
\setcounter{equation}0

We turn to an error estimation for the CDG methods. Write the error as
\[e = u - u_h = (u - u_I) + (u_I - u_h),\]
where $u_I \in V_h$ is the usual continuous piecewise quadratic interpolant
of the exact solution $u$. Using the scaling argument and the trace theorem, we
have the following result.

\begin{lemma}\label{lem:interpolation}
For all $v\in H^{3}(K)$ on $K\in\mathcal{T}_h$,
\begin{align*}
\|v-v_I\|_{K}+h_K\left|v-v_I\right|_{1,K}+h_K^2\left|v-v_I\right|_{2,K}&\lesssim
h^{3}_K|v|_{3,K}, \\
\left\|\nabla \left(v-v_I\right)\right\|_{0,\partial K}&\lesssim
h^{3/2}_K|v|_{3,K}.
\end{align*}
\end{lemma}

As a consequence of Lemma \ref{lem:interpolation}, we obtain the estimate
\begin{align}\label{est_inter}
\interleave u-u_I\interleave \lesssim h |u|_{3,\Omega}
\end{align}

Now, we are ready to derive a priori error estimates of the CDG methods 
when they are applied to solve the 4th-order elliptic variational inequality \eqref{EVI}.

\begin{theorem}\label{thm:1}
Assume the solution of the problem \eqref{EVI} satisfies $u\in H^{3}(\Omega)$
and the assumptions in Lemma \ref{lem:sta} hold.
Let $B_{h}=B_{h}^{(j)}$ with $j=1,\cdots,5$, and $u_h\in K_{h}$ be
the solution of \eqref{dcdg}. Then we have the optimal order error estimate
\begin{equation}
\interleave u-u_h\interleave\lesssim h\|u\|_{3,\Omega}.
\label{bound1}
\end{equation}
\end{theorem}
\begin{proof}
Recall the boundedness and stability of the bilinear form $B_{h}$. We have
\begin{equation}
\interleave u_I-u_h\interleave^2 \lesssim
B_{h}(u_I-u_h,u_I-u_h)\equiv T_1+T_2, \label{err1}
\end{equation}
where
\begin{align*}
T_1 &= B_{h}(u_I-u,u_I-u_h),\\
T_2&= B_{h}(u-u_h,u_I-u_h).
\end{align*}
We bound $T_1$ as follows:
\begin{equation}
  T_1 \lesssim\interleave u_I-u\interleave \interleave u_I-u_h\interleave
   \lesssim \epsilon \interleave u_I-u_h\interleave^2
    +\frac{1}{4\epsilon}\interleave u_I-u\interleave^2,
\label{err2}
\end{equation}
where $\epsilon > 0$ is an arbitrarily small number.

Since $\llbracket\nabla u\rrbracket =0$ on $e \in {\cal E}_h^i$, we
use the definition (\ref{moment}) to obtain
\begin{align*}
B_{h}(u,u_I-u_h) =&\int_{\Omega}(1-\kappa)\nabla^{2}u:\nabla_h^{2}
(u_I-u_h)\, dx +\int_{\Omega}\kappa\,{\rm
tr}\left(\nabla^{2}u\right) {\rm
tr}\left(\nabla_h^{2}(u_I-u_h)\right) dx\\
{} &-\int_{\mathcal{E}_h^0}\llbracket\nabla (u_I-u_h)\rrbracket:
\left((1-\kappa)\nabla^2u
 +\kappa\,{\rm tr}\left(\nabla^2u\right)\bI
 \right)ds\\
=&-\sum_{K\in\mathcal{T}_{h}}\int_{K}\bsigma:\nabla^{2}
(u_I-u_h)\,dx+\int_{\mathcal{E}_{h}^0}\llbracket\nabla
(u_I-u_h)\rrbracket:\bsigma\,ds.
\end{align*}
Noting $[\bsigma]=\bzero$ on $e\in{\cal E}_h^i$, we get by Lemma \ref{identities},
\begin{align*}
\sum_{K\in\mathcal{T}_{h}}\int_{K}\bsigma:\nabla^{2}(u_I-u_h)\,dx
&=-\sum_{K\in\mathcal{T}_{h}}\int_{K}\nabla(u_I-u_h)\cdot(\nabla\cdot \bsigma)\,dx\\
&{}\quad +\sum_{K\in\mathcal{T}_{h}}\int_{\partial K}\nabla(u_I-u_h)\cdot(\bsigma \bn_K)\,ds\\
&=-\int_{\Omega}\nabla(u_I-u_h)\cdot(\nabla\cdot\bsigma)\,dx
  +\int_{\mathcal{E}_{h}}\llbracket\nabla(u_I-u_h)\rrbracket:\bsigma\,ds.
\end{align*}
Thus,
\begin{align*}
B_{h}(u,u_I-u_h)= &\int_{\Omega}\nabla(u_I-u_h)\cdot(\nabla\cdot\bsigma)\,dx
-\int_{\Gamma_2\cup\Gamma_3}\llbracket\nabla (u_I-u_h)\rrbracket:\bsigma\,ds\\
= & \int_{\Omega}\nabla(u_I-u_h)\cdot(\nabla\cdot\bsigma)\,dx 
-\int_{\Gamma_2\cup \Gamma_3} \sigma_\tau\ \partial_\tau (u_I - u_h)\,ds\\
= & - \int_\Omega \nabla\cdot(\nabla\cdot\bsigma) (u_I-u_h)\,dx 
+ \langle (\nabla\cdot\bsigma)\cdot\bnu,u_I-u_h \rangle_{1/2,\Gamma}\\
&-\int_{\Gamma_2\cup \Gamma_3} \sigma_\tau \partial_\tau (u_I - u_h)\,ds.
\end{align*}

By \eqref{relation5} and \eqref{relation1}, we have
\begin{align}
B_{h}(u,u_I-u_h)=& -\int_{\Omega}\nabla\cdot(\nabla\cdot\bsigma)(u_I-u_h)\,dx
-\int_{\Gamma_3}g\ \lambda\ u_I\,ds + \int_{\Gamma_3}g\ \lambda\ u_h\,ds \nonumber\\
=&\int_{\Omega}f(u_I-u_h)\,dx -\int_{\Gamma_3}g\ \lambda\ u_I\,ds + \int_{\Gamma_3}g\ \lambda\ u_h\,ds. \label{err3}
\end{align}
Let $v_h = u_I$ in (\ref{dcdg}),
\begin{equation}\label{err4}
  B_{h}(u_h,u_I-u_h) + j(u_I) - j(u_h) \ge (f,u_I-u_h).
\end{equation}

Combining (\ref{err4}) and (\ref{err3}), and with the use of \eqref{Lar2}, we can
bound $T_2=B_h(u-u_h,u_I-u_h)$ as follows:
\begin{align*}
T_2 &\le -\int_{\Gamma_3}g\ \lambda\ u_I\,ds + \int_{\Gamma_3}g\ \lambda\ u_h\,ds
+ j(u_I) - j(u_h)\\
&= \int_{\Gamma_3}g (|u_I|-\lambda\ u_I)\,ds + \int_{\Gamma_3}g (\lambda\ u_h - |u_h|)\,ds\nonumber\\
&\le \int_{\Gamma_3}g (|u_I|-\lambda\ u_I)\,ds = \int_{\Gamma_3}g (|u_I| - |u| + \lambda\ u -\lambda\ u_I)\,ds \\
&\le 2\int_{\Gamma_3}g\ |u_I - u|\, ds \le 2\|g\|_{0,\Gamma_3}\|u_I - u \|_{0,\Gamma_3} .
\end{align*}
Hence,
\begin{equation}
T_2\lesssim h^{2}\|u\|_{3,\Omega}
\label{err5}
\end{equation}

Combining \eqref{err1}, \eqref{err2}, and (\ref{err5}), and
applying Lemma \ref{lem:interpolation}, we have
\begin{equation}\label{uiuh}
\interleave u_I-u_h\interleave^2 \lesssim h^{2}
\|u\|_{3,\Omega}^2.
\end{equation}
Finally, from the triangle inequality $\interleave u-u_h\interleave
\le \interleave u-u_I\interleave +\interleave u_I-u_h\interleave$,
\eqref{est_inter} and (\ref{uiuh}), we obtain the error bound.
\end{proof}

\section{Numerical Results}
\setcounter{equation}0

In this section, we present a numerical example with the five CDG schemes studied
in solving the elliptic variational inequality \eqref{EVI}. Let
$\Omega=(-1,1)\times(-1,1)$, $\kappa=0.3$. A generic point in $\overline{\Omega}$
is denoted as $\mathbf{x}=(x,y)^T$.  The Dirichlet boundary is
$\Gamma_1=(-1,1)\times\{ 1\}$, and the free boundary is
$\Gamma_2=\{\{-1\}\times (-1,1)\} \cup \{\{1\}\times (-1,1)\}$.
On the friction boundary $\Gamma_3=(-1,1)\times\{-1\}$, we choose $g = 1$.
The right hand side function is $f(\mathbf{x}) = 24(1-x^2)^2+24(1-y^2)^2+32(3x^2-1)(3y^2-1)$.

For a discretization of the variational inequality \eqref{EVI}, we use uniform
triangulations $\{\mathcal{T}_h\}$ of
the region $\overline{\Omega}$, and define the finite element spaces to be
\begin{align*}
V_h&:=\{v_h\in H^1_{\Gamma_1}(\Omega):v_h|_K\in P_2(K)\ \forall\, K\in\mathcal{T}_h\},\\
\bSigma_h&:=\left\{\btau_h\in \left(L^2(\Omega)\right)_s^{2\times2}:
\tau_{h,ij}|_K\in P_1(K)\ \forall\,K\in\mathcal{T}_h,\ i,j=1,2\right\}.
\end{align*}
Any function $v^h\in V_h$ can be expressed as
\[ v^h(\mathbf{x})=\sum v_i\phi_i(\mathbf{x}), \]
where $v_i=v^h(\mathbf{x}_i)$, $\{\mathbf{x}_i\}$ are the nodal points, and $\{\phi_i\}$
are the standard nodal basis functions of the space $V_h$.  The basis functions
satisfy the relation $\phi_i(\mathbf{x}_j) =\delta_{ij}$, $\delta_{ij}$ being the Kronecher delta.
The functional $j(\cdot)$ is approximated through numerical integration:
\[ j_h(v^h)=S^{\Gamma_3}_n(g\,|v^h|)=\sum w_j g(\mathbf{x}_j)|\sum v_i\phi_i(\mathbf{x}_j)|
  =\sum |w_jg(\mathbf{x}_j)\,v_j|, \]
where the summations extend to all the finite element nodes on $\overline{\Gamma_3}$,
and $S^{\Gamma_3}_n$ denotes the composite Simpson's rule using these finite element nodes.
Then the discrete problem is
\begin{equation}
\min_{u^h\in V^h}\frac{1}{2}\,a(u^h,u^h)+j_h(u^h)-(f,u^h).\label{DOPT}
\end{equation}
The matrix/vector form of the discrete optimization problem is
\begin{equation}
\min_\mathbf{u}\frac{1}{2}\mathbf{u}^T\mathbf{Au}+\|\mathbf{Bu}\|_{\ell_1}-\mathbf{u}^T\mathbf{f},
\label{MOPT}
\end{equation}
where $\mathbf{u}=(u_i)^T$, $\mathbf{A}=\left(a(\phi_i,\phi_j)\right)$,
$\mathbf{B}=\left(w_ig(\mathbf{x}_i)\delta_{ij}\right)$, and $\mathbf{f}=\left((f,\phi_j)\right)^T$.

To solve the discrete problem \eqref{MOPT}, we use the following primal-dual fixed point iteration Algorithm \ref{alg1} proposed in \cite{zhang11}. 
\begin{algorithm}
\caption{Primal Dual Fixed Point Algorithm}
\label{alg1}
\begin{algorithmic}
\STATE {Initialize $\mathbf{u}_0$ and $\mathbf{v}_0$, set parameters $\lambda\in\interval({0,\frac{1}{\lambda_{\max}(BB^T)}}],\gamma\in\interval({0,\frac{2}{\|A\|_2}})$}
\FOR{$i=1,2,3,\cdots$}
\STATE $\mathbf{u}_{k+\frac{1}{2}}=\mathbf{u}_k-\gamma(\mathbf{Au}_k-\mathbf{f})$,
\STATE $\mathbf{v}_{k+1}=(I-\mbox{prox}_{\frac{\gamma}{\lambda}\|\cdot\|_1})(\mathbf{Bu}_{k+\frac{1}{2}}+(I-\lambda \mathbf{BB}^T)\mathbf{v}_k)$,
\STATE $\mathbf{u}_{k+1}=\mathbf{u}_{k+\frac{1}{2}}-\lambda\mathbf{B}^T\mathbf{v}_{k+1}$
\ENDFOR
\end{algorithmic}
\end{algorithm}
Here for a given function $F$ of a vector variable $\mathbf{x}$,
the proximal operator $\mbox{prox}_F$ is defined as
\[\mbox{prox}_F(\mathbf{x})=\arg\min_{\mathbf{y}} F(\mathbf{y})+\frac{1}{2}\|\mathbf{x}-\mathbf{y}\|^2_2.\]
For $F=\frac{\gamma}{\lambda}\|\cdot\|_1$, the proximal operator has the explicit form (applied
to each component of the vector variable):
\[\mbox{prox}_{\frac{\gamma}{\lambda}\|\cdot\|_1}x
=\mbox{sgn}(x)\,\max\!\left(|x|-\frac{\gamma}{\lambda},0\right)
=\mbox{sgn}(x)\left(|x|-\frac{\gamma}{\lambda}\right)_+.\]

Tables \ref{table:1}--\ref{table:5} provide numerical solution errors in
the energy norm $\interleave \cdot\interleave$ and $H^1(\Omega)$ seminorm for the five DG methods
discussed in this paper.   Since the true solution of the variational inequality \eqref{EVI}
is not known, we use the numerical solution corresponding to the meshsize
$h=1/64$ as the true solution in computing the errors.  We observe that the numerical
convergence orders in the energy norm are around one, agreeing with the
theoretical error estimate \eqref{bound1}.  We note that the numerical
convergence orders in the $H^1(\Omega)$-seminorm are also close to one.

\begin{table}[H]
\caption{Error for $C^0$ IP method \eqref{form1b}}
\begin{center}
\begin{tabular}{|c|c|c|c|c|c|c|}
\hline
\multirow{2}{*}{$h$} &
\multicolumn{3}{c|}{$\interleave u-u_h\interleave$} &
\multicolumn{3}{c|}{$|u-u_h|_{H^1(\Omega)}$} \\
\cline{2-7} & $\eta=1$ & $\eta=10$ & $\eta=100$  & $\eta=1$ & $\eta=10$ & $\eta=100$ \\ \hline
1/2 & $5.1859$ & $4.2973$ & $3.1164$ & $0.9367$ & $0.8126$ & $0.5635$\\ \hline
1/4 & $3.3677$ & $2.6726$ & $1.5923$ & $0.5105$ & $0.4135$ & $0.3133$\\ \hline
1/8 & $1.8625$ & $1.4407$ & $0.8122$ & $0.2691$ & $0.2049$ & $0.1580$\\ \hline
1/16 & $0.8601$ & $0.7652$ & $0.4422$ & $0.1355$ & $0.1059$ & $0.0801$\\ \hline
% 1/32 & $0.2548$ & $0.2509$ & $0.2083$ & $0.0174$ & $0.0171$ & $0.0136$\\ \hline
\end{tabular}
\end{center}
\label{table:1}
\end{table}

\begin{table}[H]
\caption{Error for NIPG method \eqref{form2b}}
\begin{center}
\begin{tabular}{|c|c|c|c|c|c|c|}
\hline
\multirow{2}{*}{$h$} &
\multicolumn{3}{c|}{$\interleave u-u_h\interleave$} &
\multicolumn{3}{c|}{$|u-u_h|_{H^1(\Omega)}$} \\
\cline{2-7} & $\eta=1$ & $\eta=10$ & $\eta=100$  & $\eta=1$ & $\eta=10$ & $\eta=100$ \\ \hline
1/2 & $5.5411$ & $4.4659$ & $3.1178$ & $1.0000$ & $0.8250$ & $0.6638$\\ \hline
1/4 & $3.6029$ & $2.7708$ & $1.6071$ & $0.5699$ & $0.4233$ & $0.3607$\\ \hline
1/8 & $1.9137$ & $1.5359$ & $0.7961$ & $0.2829$ & $0.2246$ & $0.1853$\\ \hline
1/16 & $0.9485$ & $0.7594$ & $0.3929$ & $0.1491$ & $0.1144$ & $0.0934$\\ \hline
% 1/32 & $0.1472$ & $0.1192$ & $0.1025$ & $0.0082$ & $0.0077$ & $0.0076$\\ \hline
\end{tabular}
\end{center}
\label{table:2}
\end{table}

\begin{table}[H]
\caption{Error for Wells-Dung DG formulation \eqref{form3b}}
\begin{center}
\begin{tabular}{|c|c|c|c|c|c|c|}
\hline
\multirow{2}{*}{$h$} &
\multicolumn{3}{c|}{$\interleave u-u_h\interleave$} &
\multicolumn{3}{c|}{$|u-u_h|_{H^1(\Omega)}$} \\
\cline{2-7} & $\eta=1$ & $\eta=10$ & $\eta=100$  & $\eta=1$ & $\eta=10$ & $\eta=100$ \\ \hline
1/2 & $4.4617$ & $3.4573$ & $2.0932$ & $0.8062$ & $0.7650$ & $0.3785$\\ \hline
1/4 & $2.8185$ & $2.2331$ & $1.3618$ & $0.4301$ & $0.3885$ & $0.2486$\\ \hline
1/8 & $1.4545$ & $1.1473$ & $0.6794$ & $0.2131$ & $0.2035$ & $0.1253$\\ \hline
1/16 & $0.7322$ & $0.6270$ & $0.3832$ & $0.1085$ & $0.1036$ & $0.0645$\\ \hline
% 1/32 & $0.3178$ & $0.3032$ & $0.2990$ & $0.0205$ & $0.0402$ & $0.0167$\\ \hline
\end{tabular}
\end{center}
\label{table:3}
\end{table}

\begin{table}[H]
\caption{Error for Baker-DG formulation \eqref{form4b}}
\begin{center}
\begin{tabular}{|c|c|c|c|c|c|c|}
\hline
\multirow{2}{*}{$h$} &
\multicolumn{3}{c|}{$\interleave u-u_h\interleave$} &
\multicolumn{3}{c|}{$|u-u_h|_{H^1(\Omega)}$} \\
\cline{2-7} & $\eta=1$ & $\eta=10$ & $\eta=100$  & $\eta=1$ & $\eta=10$ & $\eta=100$ \\ \hline
1/2 & $4.8538$ & $4.1180$ & $2.0977$ & $0.8771$ & $0.8360$ & $0.3793$\\ \hline
1/4 & $2.8524$ & $2.4987$ & $1.3632$ & $0.4662$ & $0.4625$ & $0.2489$\\ \hline
1/8 & $1.5067$ & $1.2842$ & $0.6765$ & $0.2447$ & $0.2409$ & $0.1255$\\ \hline
1/16 & $0.7629$ & $0.6747$ & $0.3842$ & $0.1259$ & $0.1212$ & $0.0657$\\ \hline
%1/32 & $0.2846$ & $0.2512$ & $0.2443$ & $0.0196$ & $0.0172$ & $0.0161$\\ \hline
\end{tabular}
\end{center}
\label{table:4}
\end{table}

\begin{table}[H]
\caption{Error for LCDG method \eqref{form5b}}
\begin{center}
\begin{tabular}{|c|c|c|c|c|c|c|}
\hline
\multirow{2}{*}{$h$} &
\multicolumn{3}{c|}{$\interleave u-u_h\interleave$} &
\multicolumn{3}{c|}{$|u-u_h|_{H^1(\Omega)}$} \\
\cline{2-7} & $\eta=1$ & $\eta=10$ & $\eta=100$  & $\eta=1$ & $\eta=10$ & $\eta=100$ \\ \hline
1/2 & $4.6407$ & $4.2599$ & $2.5863$ & $0.8384$ & $0.7696$ & $0.5579$\\ \hline
1/4 & $2.8265$ & $2.2147$ & $1.6213$ & $0.4546$ & $0.4111$ & $0.2863$\\ \hline
1/8 & $1.5011$ & $1.2460$ & $0.8669$ & $0.2311$ & $0.2301$ & $0.1453$\\ \hline
1/16 & $0.7517$ & $0.6341$ & $0.4705$ & $0.1189$ & $0.1186$ & $0.0812$\\ \hline
%1/32 & $0.3147$ & $0.2052$ & $0.1294$ & $0.0201$ & $0.0131$ & $0.0083$\\ \hline
\end{tabular}
\end{center}
\label{table:5}
\end{table}

\end{document}